\documentclass[10pt,english,reqno]{amsart}

\usepackage{amsthm,amsmath,amssymb,amsfonts}
\usepackage{graphicx}
\usepackage[font=scriptsize,skip=6pt]{caption}

\usepackage[
  pdfauthor={},
  pdfkeywords={},
  pdftitle={},
  pdfcreator={},
  pdfproducer={},
  linktocpage,colorlinks,bookmarksnumbered]{hyperref}

\newdimen\papertrimheight
\newdimen\papertrimwidth
\theoremstyle{plain}
\newtheorem{theorem}{Theorem}
\newtheorem{lemma}{Lemma}

\theoremstyle{definition}
\newtheorem{definition}{Definition}
\newtheorem{example}{Example}

\DeclareMathOperator{\DefEq}{\stackrel{{\rm def}}{=}}
\DeclareMathOperator{\DefEqDisp}{\hspace{0.15em}\DefEq\hspace{0.15em}}
\DeclareMathOperator{\denom}{denom}
\DeclareMathOperator{\lcm}{lcm}

\newcommand{\N}{\mathbb{N}}
\newcommand{\Z}{\mathbb{Z}}
\newcommand{\Q}{\mathbb{Q}}

\newcommand{\PSet}{\mathcal{P}}
\newcommand{\DD}{\mathfrak{D}}

\newcommand{\doi}[1]{\href{http://dx.doi.org/#1}{\texttt{DOI:\,#1}}}

\begin{document}

\title{Power-Sum Denominators}
\markright{Power-Sum Denominators}
\author{Bernd C. Kellner and Jonathan Sondow}
\address{G\"ottingen, Germany}
\email{bk@bernoulli.org}
\address{New York, USA}
\email{jsondow@alumni.princeton.edu}

\begin{abstract}
The {\em power sum} $1^n + 2^n + \cdots + x^n$ has been of interest to
mathematicians since classical times. Johann Faulhaber, Jacob Bernoulli, and
others  who followed expressed power sums as polynomials in $x$ of degree $n+1$
with rational coefficients. Here we consider the denominators of these
polynomials, and prove some of their properties. A remarkable one is that such
a denominator equals $n+1$ times the squarefree product of certain primes $p$
obeying the condition that the sum of the base-$p$ digits of $n+1$ is at
least~$p$. As an application, we derive a squarefree product formula for the
denominators of the Bernoulli polynomials.
\end{abstract}

\maketitle

\vspace*{-2.5em}
\begin{figure}[htbp]
  \begin{center}
    \includegraphics[width=2in]{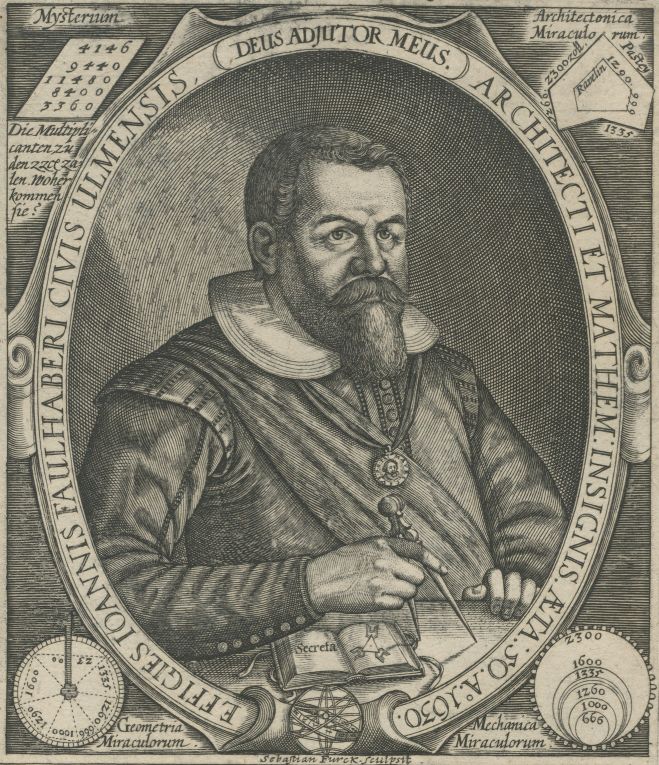}
    \caption[Cap]{Johann Faulhaber (1580--1635). \protect\footnotemark}
    \label{FIG:Faulhaber}
  \end{center}
\end{figure}
\vspace*{-2.5em}

\footnotetext{\copyright\ Copyright and used with permission by the Archive of
the City of Ulm \cite{furck}.}


\section{Introduction}
Johann Faulhaber was ``known in his day as The Great Arithmetician of Ulm''
(see \cite[p.~106]{cg}, \cite[p.~152]{stopple}). In his 1631 book
{\em Academia Algebrae}~\cite{faul1631}, Faulhaber worked out formulas
for {\em power sums} $1^n + 2^n + \dotsb + x^n$ as polynomials
in $x$ of degree $n+1$ with rational coefficients. He found that
\begin{align*}
  1^0 + 2^0 + \dotsb + x^0 &= x,\\
  1^1 + 2^1 + \dotsb + x^1 &= \frac{1}{2} (x^2+x),\\
  1^2 + 2^2 + \dotsb + x^2 &= \frac{1}{6} (2x^3+3x^2+x),\\
\intertext{}
  1^3 + 2^3 + \dotsb + x^3 &= \frac{1}{4} (x^4+2x^3+x^2),\\
  1^4 + 2^4 + \dotsb + x^4 &= \frac{1}{30} (6x^5+15x^4+10x^3-x),\\
  1^5 + 2^5 + \dotsb + x^5 &= \frac{1}{12} (2x^6+6x^5+5x^4-x^2),
\end{align*}
and so on up to $n=17$ (see~\cite{edwards, knuth, mactutor, sk};
a comprehensive survey of Faulhaber's life and mathematical work is given
in \cite{schneider1993}).
The fractions in these formulas naturally lead one to consider the denominators.

\begin{definition} \label{DEF:d_n}
For $n \geq 0$, the $n$th {\em power-sum denominator} is the smallest positive
integer $d_n$ such that $d_n \cdot (1^n + 2^n + \dotsb + x^n)$ is a polynomial
in $x$ with {\em integer} coefficients.
\end{definition}

The first few values of $d_n$ (see \cite[Sequence A064538]{oeis}) are
$$
  d_n = 1, 2, 6, 4, 30, 12, 42, 24, 90, 20, 66, 24, 2730, 420, 90, 48, 510, \dotsc.
$$

In this article, we study the power-sum denominators and related numbers and
prove some of their properties. We first collect some fairly straightforward
ones. Throughout the paper, $p$ always denotes a prime.

\begin{theorem} \label{THM:main}
The sequence of power-sum denominators $d_n$ for $n \geq 0$ has the following
properties:

\begin{enumerate}
\item $p \mid d_n \implies p \leq n+1$.

\item $d_n$ is divisible by $n+1$, and we have a squarefree quotient
   $$
     q_n \DefEq \frac{d_n}{n+1}.
   $$

\item $d_n$ is even for all $n \geq 1$, while $q_n$ is odd if and only
      if $n=2^r-1$ for some $r \geq 0$.

\item
  $
    p \mid q_n \implies p \leq M_n
    \DefEqDisp
    \begin{cases}
      \, \dfrac{n+2}{2}, &\text{if $n$ is even,} \\[0.6em]
      \, \dfrac{n+2}{3}, &\text{if $n$ is odd.}
    \end{cases}
  $
\end{enumerate}
\end{theorem}

The first few quotients (see \cite[Sequence A195441]{oeis}) are
$$
  q_n = 1, 1, 2, 1, 6, 2, 6, 3, 10, 2, 6, 2, 210, 30, 6, 3, 30, 10, 210, 42, 330, \dotsc.
$$
Their values can be computed from the following surprising and remarkable
formula. As usual, an empty product is defined to be $1$.

\begin{theorem} \label{THM:formula}
For all $n \geq 0$, we have the prime factorization
\begin{equation} \label{EQ:prod-q_n}
  q_n = \hspace{-1.0em} \prod_{
      \substack{
        p \, \leq \, M_n\\[0.2em]
        s_p(n+1) \, \geq \, p}
      }
    \hspace{-1.0em} p,
\end{equation}
where $q_n$ and $M_n$ are as in Theorem~\ref{THM:main}, and $s_p(n)$ denotes
the sum of the base-$p$ digits of~$n$, as defined in Section~\ref{SEC:proofs2}.
Moreover, the bound $M_n$ is sharp for infinitely many even (respectively, odd)
values of $n$. In particular, the sequence $(q_n)_{n\ge0}$ is unbounded.
\end{theorem}

\pagebreak

\begin{example}
To illustrate Theorems~\ref{THM:main} and~\ref{THM:formula}, we compute the table
\begin{center}
\begin{tabular}[c]{|c|c|l|l|}\hline
  $n$ & $M_n$ & \multicolumn{1}{c|}{$q_n$} & \multicolumn{1}{c|}{$d_n$} \\\hline\hline
  $19$ &  $7$ & $2 \cdot 3 \cdot 7 = 42$  & $20 \cdot q_{19} = 840$ \\\hline
  $20$ & $11$ & $2 \cdot 3 \cdot 5 \cdot 11= 330$ & $21 \cdot q_{20} = 6930$ \\\hline
\end{tabular}
\end{center}

\noindent which depends on the values of $s_p(n+1)$ given by

\begin{center}
\begin{tabular}{|c|c|c|c|c|}\hline
  $p$       & $2$ & $3$ & $5$ & $7$ \\\hline
  $s_p(20)$ & $2$ & $4$ & $4$ & $8$ \\\hline
\end{tabular}
\quad and \quad
\begin{tabular}{|c|c|c|c|c|c|}\hline
  $p$       & $2$ & $3$ & $5$ & $7$ & $11$ \\\hline
  $s_p(21)$ & $3$ & $3$ & $5$ & $3$ & $11$ \\\hline
\end{tabular}
.
\end{center}
\end{example}

The rest of the paper is organized as follows. The next section is devoted to
preliminaries, including Bernoulli's formula for power sums and the von
Staudt--Clausen theorem on Bernoulli numbers. In Section~\ref{SEC:lemmas}
we consider some properties of the binomial coefficients and we prove five lemmas.
Section~\ref{SEC:proofs} contains the proof of Theorem~\ref{THM:main}.
In Section~\ref{SEC:properties} we use a congruence on binomial coefficients
due to Hermite and Bachmann to give another formula for the quotients $q_n$.
In Section~\ref{SEC:proofs2} we prove Theorem~\ref{THM:formula} using
$p$-adic methods, including results of Legendre and Lucas. In the final section
the theorems are applied to the denominators of the Bernoulli polynomials.


\section{Preliminaries}
In his {\em Ars Conjectandi} \cite[pp.~96--98]{bern} of 1713, Jacob Bernoulli
generalized Faulhaber's formulas, but without giving a rigorous proof. Later a
proof followed as a special case of the more general {\em Euler--Maclaurin
summation formula}, which was independently found (cf.~\cite[p.~402]{grab}) by
Euler \cite[pp.~17--18]{euler} and Maclaurin \cite[pp.~676--677]{macl} around 1735.
In modern terms, {\em Bernoulli's formula} for power sums
(see \cite[pp.~106--109]{cg}) states that for $n \geq 1$,
\begin{equation} \label{EQ:def-S_n}
  S_n(x)
  \DefEqDisp 1^n + 2^n + \dotsb + (x-1)^n
  = \frac{B_{n+1}(x) - B_{n+1}}{n+1},
\end{equation}
where the $n$th {\em Bernoulli polynomial} $B_n(x)$ is defined symbolically as
\begin{equation} \label{EQ:def-B_nx}
  B_n(x)
  \DefEqDisp (B+x)^n
  \DefEqDisp \sum_{k=0}^{n} \binom{n}{k} B_k \, x^{n-k}
\end{equation}
and the {\em Bernoulli numbers} $B_0, B_1, B_2, \dotsc$ are rational numbers
defined by the generating function
$$
  \frac{t}{e^t - 1} = \sum_{k=0}^{\infty} B_k \frac{t^k}{k!} \quad (|t| < 2\pi).
$$

It turns out that $B_n=0$ for odd $n > 1$ (see, e.g., \cite[Section 7.9]{hw}).
The sequence of nonzero Bernoulli numbers starts with
\begin{equation} \label{EQ:B0B1}
  B_0 = 1, \quad B_1 = -\frac{1}{2},
\end{equation}
and continues with
$$
  B_2 = \frac{1}{6},\; B_4 = -\frac{1}{30},\; B_6 = \frac{1}{42},\;
  B_8 = -\frac{1}{30},\; B_{10} = \frac{5}{66},\; B_{12} = -\frac{691}{2730},\; \dotso.
$$

\begin{definition} \label{DEF:denom}
Every rational number $\rho \in \Q$ can be written uniquely in lowest terms as
a fraction $\rho = \nu / \delta$ with $\nu \in \Z$ and $\delta \in \N$. We define
$\denom (\rho)\DefEq \delta$. (In particular, if $m \in \Z$, then $\denom(m) = 1$.)
This definition extends uniquely to polynomials \mbox{$p(x) \in \Q[x]$} by
defining $\denom \bigl( p(x) \bigr)$ to be the smallest positive integer $d$
such that \mbox{$d \cdot p(x) \in \Z[x]$}.
\end{definition}

A fundamental property of the nonzero Bernoulli numbers $B_n$ is the simple shape
of their denominators. By the famous {\em von Staudt--Clausen theorem}
(cf.~\cite[p.~109]{cg}, \cite[Section 7.9]{hw}),
independently found by von Staudt~\cite{staudt} and Clausen~\cite{clausen} in 1840,
\begin{equation} \label{EQ:denom-B_n}
  \denom ( B_n ) =
  \hspace{-0.3em}
  \prod_{(p-1) \, \mid \, n}
  \hspace{-0.3em} p
  \quad (n \in 2\N),
\end{equation}
where $n$ is even and the product runs over all primes $p$ such that $p-1$
divides $n$ (thus always including the primes $2$ and $3$).
Together with $B_0$ and $B_1$, all nonzero Bernoulli numbers have
a squarefree denominator.

The {\em integer-valued} polynomial $S_n(x) \in \Q[x]$ satisfies the functional
equation
\begin{equation} \label{EQ:func-S_n}
  S_n(x+1) = S_n(x) + x^n = 1^n + 2^n + \dotsb + x^n \quad (x \in \N).
\end{equation}
Definitions~\ref{DEF:d_n} and~\ref{DEF:denom} yield that $d_n = \denom(S_n(x+1))$.
As we now show, it is also true that
\begin{equation} \label{EQ:denom-S_n}
  d_n = \denom \bigl( S_n(x) \bigr).
\end{equation}
This enables a link to Bernoulli's formula \eqref{EQ:def-S_n} and the Bernoulli
polynomials.

The case $n=0$ of \eqref{EQ:denom-S_n} is easily seen directly.
For $n \geq 1$, we obtain from formulas \eqref{EQ:def-S_n}, \eqref{EQ:def-B_nx},
and \eqref{EQ:B0B1} that
\begin{equation} \label{EQ:first-part-S_n}
  S_n(x) = \frac{x^{n+1}}{n+1} - \frac{x^n}{2}  + \dotsb.
\end{equation}
Comparison with \eqref{EQ:func-S_n} shows that $S_n(x)$ differs from $S_n(x+1)$
by the summand $x^n$, which only results in a sign change from $- \frac{1}{2}$
to $+ \frac{1}{2}$ in the coefficient of $x^n$ in \eqref{EQ:first-part-S_n}.
This sign change has no effect on the denominators of the polynomials in question,
so \eqref{EQ:denom-S_n} holds. (All of this will be shown in more detail in the
proof of Theorem~\ref{THM:main}.)

Revisiting the formulas of Faulhaber on the first page,
one might notice the simple pattern $\frac{1}{d_n} \times \text{polynomial}$.
The next lemma clarifies this observation that the numerator is always~$1$, as
a supplement to our study of power-sum denominators. (We omit the trivial
case $n=0$ here, since $S_n(x)$ is defined for $n \geq 1$.)

\begin{lemma}
For $n \geq 1$,  we have
$$
  S_n(x) = \frac{1}{d_n} \cdot p_n(x),
$$
where the coefficients of the polynomial $p_n(x) \in \Z[x]$ are coprime.
\end{lemma}

\begin{proof}
From \eqref{EQ:denom-S_n} we deduce the decomposition
$$
  S_n(x) = \frac{a_n}{d_n} \cdot p_n(x),
$$
where $a_n$ and $d_n$ are coprime positive integers, and $p_n(x) \in \Z[x]$ has
coprime coefficients. Since $S_n(x)$ is integer-valued, we infer that
$S_n(x) \equiv 0 \pmod{a_n}$ for $x = 1, 2, \dotsc$. In particular,
$1 = S_n(2) \equiv 0 \pmod{a_n}$, which implies $a_n = 1$.
\end{proof}

The lemma confirms the observation that, in Faulhaber's formulas on
the first page, the integer coefficients of the $n$th polynomial are coprime.
Moreover, their sum must equal the denominator $d_n$, as one sees by setting $x=1$.


\section{Lemmas} \label{SEC:lemmas}
Before we give proofs of our theorems, we need several more lemmas.

\begin{lemma} \label{LEM:poly-T_n}
For $n \geq 1$, define
$$
  T_n(x) \DefEqDisp (n+1) \frac{S_n(x)}{x}.
$$
Then $T_n(x)$ is a monic polynomial in $\Q[x]$, and is given by
$$
  T_n(x) = \sum_{k=0}^{n} \binom{n+1}{k} B_k \, x^{n-k}.
$$
\end{lemma}

\begin{proof}
From \eqref{EQ:def-S_n} and \eqref{EQ:def-B_nx}, it follows that
$$
  T_n(x) = \frac{B_{n+1}(x) - B_{n+1}}{x}
         = \sum_{k=0}^{n} \binom{n+1}{k} B_k \, x^{n-k} \in \Q[x].
$$
Since the coefficient of $x^n$ is $\binom{n+1}{0} B_0 = 1$ by \eqref{EQ:B0B1},
the polynomial $T_n(x)$ is monic.
\end{proof}

\begin{lemma} \label{LEM:pset_def}
For any integers $m \geq k \geq 0$, we compute the denominators
$$
  D_{m,k} \DefEqDisp \denom \left( \binom{m}{k} B_k \right)
  = \hspace{-0.3em} \prod_{p \, \in \, \PSet_{m,k}} \hspace{-0.3em} p,
$$
where the sets $\PSet_{m,k}$ of primes are defined by the following cases:
\begin{enumerate}
\item $\PSet_{m,k} \DefEqDisp \emptyset$, if $k=0$ or $k \geq 3$ is odd.

\item $\PSet_{m,k} \DefEqDisp
      \begin{cases}
        \, \emptyset, &\text{if $k=1$ and $m$ is even}, \\
        \, \{ 2 \},   &\text{if $k=1$ and $m$ is odd}.
      \end{cases}$

\item $\PSet_{m,k} \DefEqDisp \left\{ p : (p-1) \mid k \ \text{and}\
      p \nmid \binom{m}{k} \right\}$, if $k \geq 2$ is even.
\end{enumerate}
\end{lemma}

\begin{proof}
Recall that $B_0 = 1$ and $B_1 = -\frac{1}{2}$ by \eqref{EQ:B0B1},
and that $B_k = 0$ for odd $k > 1$.

(i). It follows that $D_{m,0} = 1$, and $D_{m,k} = 1$ if $k \geq 3$ is odd.
This shows case (i).

(ii). We have $D_{m,1} = \denom\left( -\frac{m}{2} \right)$, so $D_{m,1} = 1$
if $m$ is even, while $D_{m,1} = 2$ if $m$ is odd. Case (ii) follows.

(iii). By the von Staudt--Clausen theorem in \eqref{EQ:denom-B_n},
$\denom(B_k)$ is squarefree.
Thus $D_{m,k}$ is the product of the primes given by \eqref{EQ:denom-B_n}
(with $k$ in place of $n$), but excluding prime factors of $\binom{m}{k}$.
This proves case (iii).
\end{proof}

Hereafter, we will use the convention that $\max(\emptyset) \DefEq 0$.

\begin{lemma} \label{LEM:pset_odd}
For odd $m \geq 3$ and $k = 2, 4, \dotsc, m-1$, the set $\PSet_{m,k}$
defined in Lemma~\ref{LEM:pset_def} satisfies the bound
$$
  \max(\PSet_{m,k}) \leq \min \left( k+1, \frac{m+1}{2} \right).
$$
\end{lemma}

\begin{proof}
If $P_{m,k} = \emptyset$, the bound holds vacuously, and we are done.
Otherwise, fix \mbox{$p \in P_{m,k}$}. Then $(p-1) \mid k$, so the trivial bound
$p \leq k+1$ holds. It remains to show that $p \leq m_o \DefEq \frac{1}{2} (m+1)$
when $k \leq m-1$.

\textit{\mbox{Case 1}.}~
If $k < m_o$, then $p \leq k+1 \leq m_o$ and we are done.

\textit{\mbox{Case 2}.}
If $m_o \leq k \leq m-1$, assume first that $p-1 < k$.
Then the integer \mbox{$\frac{k}{p-1} \geq 2$}, so
$p-1 \leq \frac{1}{2} k \leq \frac{1}{2} (m-1) < m_o$. Thus $p \leq m_o$.

\textit{\mbox{Case 3}.}
There remains the possibility that $m_o \leq k \leq m-1$
and $p-1 = k$. But then $m-p+2 \leq p \leq m$, so the binomial coefficient
\begin{equation} \label{EQ:binomial}
  \binom{m}{k} = \binom{m}{p-1} = \frac{m (m-1) \dotsm p \dotsm (m-p+2)}{(p-1)!}
\end{equation}
is divisible by the prime $p$, contradicting $p \in \PSet_{m,k}$.
This completes the proof.
\end{proof}

\begin{lemma} \label{LEM:pset_even}
For even $m \geq 4$ and $k = 2,4,\dotsc, m-2$, the set $\PSet_{m,k}$
defined in Lemma~\ref{LEM:pset_def} satisfies the bound
$$
  \max( \PSet_{m,k} ) \leq \min \left( k+1, \frac{m+1}{3} \right).
$$
\end{lemma}

\begin{proof}
As in the proof of Lemma~\ref{LEM:pset_odd} through Case~1,
it suffices to show that $p \in \PSet_{m,k}$ implies
$p \leq m_e \DefEq \frac{1}{3} (m+1)$ when $m_e \leq k \leq m-2$.

If $p=2$, then $m=4$ would imply that $ \binom{m}{k} = \binom{4}{2} = 6$ is
divisible by $p$, contradicting $p \in \PSet_{m,k}$. Thus if $p=2$, then
$m \geq 6$, which implies $p \leq m_e$. So from now on we assume that $p$ is odd.
Set $k' \DefEq \frac{k}{p-1} \in \N$.

\textit{Case 1.}
If $k' \geq 3$, then $p-1 \leq \frac{1}{3} k \leq \frac{1}{3} (m-2) = m_e-1$,
so $p \leq m_e$.

\textit{Case 2.}
If $k' = 2$, then $2p = k+2 \leq m$. If in addition $m-2p+3 \leq p$, then
$$
  \binom{m}{k}
  = \binom{m}{2p-2}
  = \frac{m (m-1) \dotsm 2p \dotsm p \dotsm (m-2p+3)}{1 \cdot 2 \dotsm p \dotsm (2p-2)}
$$
is divisible by $p$, contradicting $p \in \PSet_{m,k}$. Hence $p < m-2p+3$,
so $3p \leq m+2$. Now, $p$ odd and $m$ even imply $3p \leq m+1$, so $p \leq m_e$.

\textit{Case 3.}
If $k' = 1$, then $p = k+1 < m$. If in addition $m - p + 2 \leq p$,
then \eqref{EQ:binomial} implies that $p \mid \binom{m}{k}$, a contradiction.
Thus $2p < m+2$. As $2p$ and $m$ are even, we therefore get $2p \leq m$.
Now, if $m_e < p$, then $m-p+1 < 2p$, so
$$
  \binom{m}{k}
  = \binom{m}{p-1}
  = \frac{m (m-1) \dotsm 2p \dotsm (m-p+2)}{(p-1)!}
$$
is divisible by $p$, contradicting  $p \in \PSet_{m,k}$. Thus $p \leq m_e$,
and we are done.
\end{proof}

For the next lemma we need some properties of binomial coefficients modulo~$2$.
Drawing {\em Pascal's triangle}~$\!\!\pmod{2}$, with a dot for the digit~$1$
and a blank for~$0$, one obtains down to row $2^r-1$ for $r \geq 2$ a dotted,
framed triangle $\Delta_r$ with a self-similar pattern (see Figures~\ref{FIG:Pascal1}
and~\ref{FIG:Pascal2} as well as \cite[Fig.~2, p.~567]{wolfram} and \cite{gran}).
Letting $r \to \infty$ while scaling to an equilateral triangle of fixed size,
this leads to a fractal, which is subdivided recursively and is called the
{\em Sierpi\'{n}ski gasket}, introduced in \cite{sier}.

\begin{figure}[htb]
  \centering
  \begin{minipage}[t]{0.49\linewidth}
    \centering
    \includegraphics[width=4.9cm]{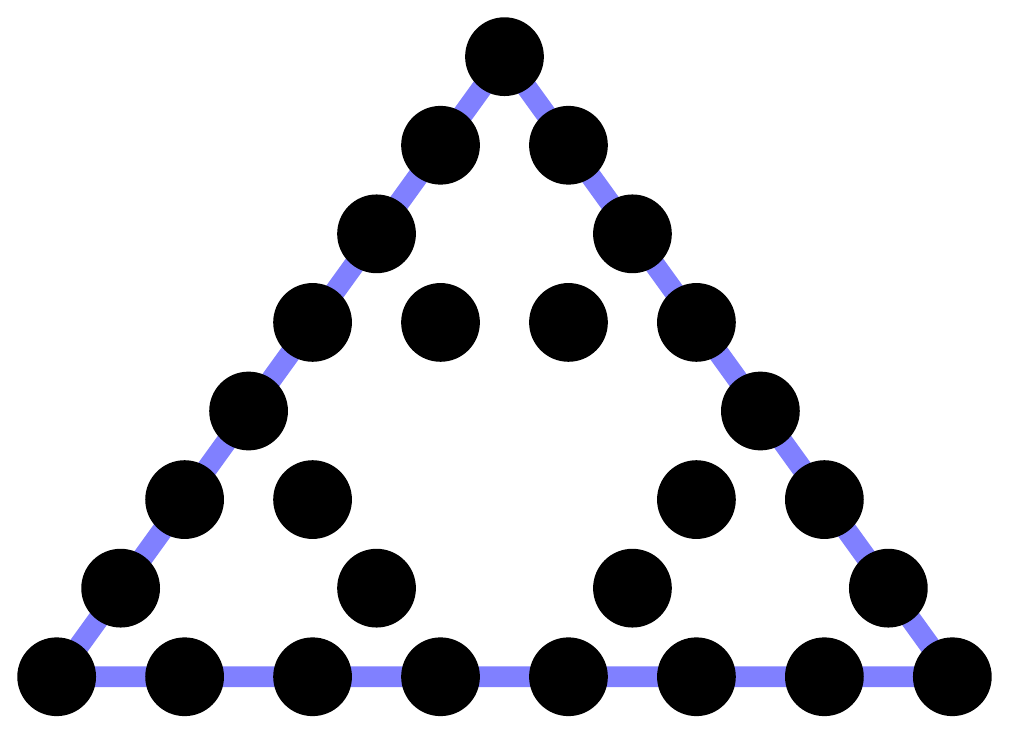}
    \caption[Cap]{Pascal's triangle~$\!\!\!\pmod{2}$:\,$\Delta_3$.}
    \label{FIG:Pascal1}
  \end{minipage}%
  \hfill
  \begin{minipage}[t]{0.49\linewidth}
    \centering
    \includegraphics[width=4.9cm]{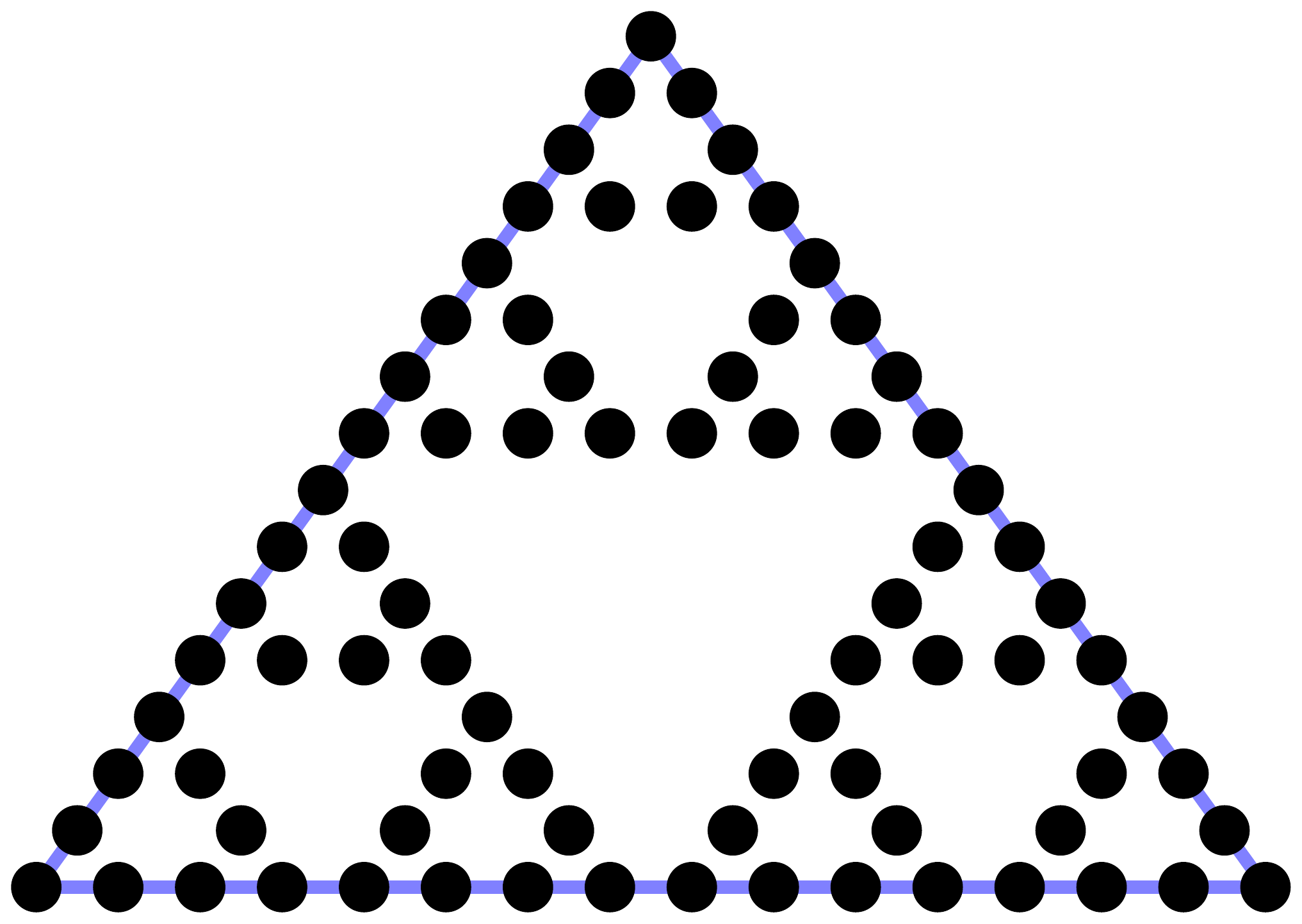}
    \caption[Cap]{Pascal's triangle~$\!\!\!\pmod{2}$:\,$\Delta_4$.}
    \label{FIG:Pascal2}
  \end{minipage}
\end{figure}

In the $m$th row of Pascal's triangle, the entries
$\binom{m}{k} \not\equiv 0 \pmod{p}$ with \mbox{$0 \leq k \leq m$}
can be counted as follows. Writing $m$ as a string
$\alpha_\ell \alpha_{\ell-1} \dotsm \alpha_0$ of base-$p$ digits
$\alpha_j$ $(0 \leq j \leq \ell)$, the number of such entries equals
$$
  \#_p(m) \DefEqDisp (\alpha_0 + 1) (\alpha_1 + 1) \dotsm (\alpha_\ell + 1).
$$
The case $p=2$, which we use below, was proved by Glaisher \cite{glaisher}, and
the general case by Fine \cite{fine}. Since $\binom{m}{0} = \binom{m}{m} = 1$,
we deduce that
\begin{equation} \label{EQ:binom-even}
\begin{split}
  \binom{m}{k} \text{\ is even} \quad (0 < k < m) \quad
  &\iff \quad \#_2(m) = 2 \\
  &\iff \quad m = 2^r \quad (r \geq 1).
\end{split}
\end{equation}

As a complement, it follows easily that $\binom{m}{0}, \binom{m}{1}, \dotsc, \binom{m}{m}$
are all odd if and only if $m=2^r-1$ $(r \geq 0)$. This explains, together with
$\binom{m}{0} = \binom{m}{m} = 1$ in each row, why the above mentioned triangle
$\Delta_r$ is always framed.

We now consider a special case, which we use later.

\begin{lemma} \label{LEM:binom-2-power}
Let $m \geq 4$ be even. The binomial coefficients
$\binom{m}{2}, \binom{m}{4}, \dotsc, \binom{m}{m-2}$
are all even if and only if $m$ is a power of~$2$.
\end{lemma}

\begin{proof}
In view of \eqref{EQ:binom-even}, it suffices to show that if $m$ is even and
$k$ is odd, then $\binom{m}{k}$ is also even.
Indeed, since $k$ is odd and so $m-(k-1)$ is even, we have
\begin{equation*}
  \binom{m}{k} =  \frac{m-(k-1)}{k} \binom{m}{k-1} \equiv 0 \pmod{2}. \qedhere
\end{equation*}
\end{proof}


\section{Proof of \texorpdfstring{Theorem~\ref{THM:main}}{Theorem~1}} \label{SEC:proofs}
We can now prove our first main result.

\begin{proof}[Proof of Theorem~\ref{THM:main}]
The case $n=0$ is trivial, with $d_0 = q_0 = M_0 = 1$ satisfying all required
properties. For the rest of the proof, we assume that $n \geq 1$.

(i), (ii).
By \eqref{EQ:denom-S_n} we have $d_n = \denom( S_n(x) )$. Combining
Lemmas~\ref{LEM:poly-T_n} and~\ref{LEM:pset_def} shows that
$$
  T_n(x) = (n+1)\frac{S_n(x)}{x}
    = \sum_{k=0}^{n} \frac{N_{n+1,k}}{D_{n+1,k}} \, x^{n-k},
$$
where the denominators $D_{n+1,k}$ are determined by Lemma~\ref{LEM:pset_def},
while the numerators $N_{n+1,k}$ are certain integers that play no role in the
proof. Since the $D_{n+1,k}$ are squarefree, the least common multiple
\begin{equation} \label{EQ:gamma_n}
  l_n \DefEqDisp \lcm( D_{n+1,1}, \dotsc, D_{n+1,n} )
\end{equation}
is also squarefree. Since $T_n(x) \in \Q[x]$ is a monic polynomial,
$l_n$ is the smallest positive integer with the property that
$$
	l_n \cdot T_n(x) \in \Z[x].
$$
Comparing this to the numbers $d_n$ and $q_n$, we observe that
$$
  d_n = (n+1) \, l_n \quad \text{and} \quad q_n = l_n.
$$
Using Lemma~\ref{LEM:pset_def} and definition \eqref{EQ:gamma_n}, we further
obtain that
\begin{equation} \label{EQ:q_n-prod}
  q_n =
  \hspace{-0.3em}
  \prod_{p \, \in \, \PSet_{n+1}}
  \hspace{-0.3em} p,
  \quad \text{where} \quad
  \PSet_n \DefEqDisp \bigcup_{k=1}^{n-1} \PSet_{n,k}.
\end{equation}
From the construction of the sets $\PSet_{n+1,k}$ (see Lemma~\ref{LEM:pset_def}),
we infer that
\begin{equation} \label{EQ:pset-1}
  \max( \PSet_{n+1} ) \leq n+1.
\end{equation}
This proves (i) and (ii).

(iii).
Let $m=n+1$. We have to show that $q_n$ is odd if and only if $m = 2^r$
with $r \geq 1$. By \eqref{EQ:q_n-prod}, we know that
$$
  2 \nmid q_n \quad \iff \quad
  2 \notin \PSet_{m,k} \quad (1 \leq k < m).
$$
Recall Lemma~\ref{LEM:pset_def}.
Since $\PSet_{m,1} = \{ 2 \}$ if $m$ is odd, and $\PSet_{m,1} = \emptyset$
otherwise, there remains the case where $m$ is even. If $m=2$, then
$\PSet_m = \emptyset$ and therefore $q_n=1$ is odd and we are done.
Now let $m \geq 4$ be even. Remember that $\PSet_{m,k} = \emptyset$ for odd
$k \geq 3$, and that if $k$ is even, then $2 \notin \PSet_{m,k}$ implies
$2 \mid \binom{m}{k}$.
With the help of Lemma~\ref{LEM:binom-2-power}, we finally deduce that
$2 \notin \PSet_{m,k}$, for $k = 2, 4, \dotsc, m-2$, if and only if $m$ is a
power of~$2$.
As a consequence, the product $m \, q_n = (n+1) \, q_n = d_n$ is always even
for $n \geq 1$. This shows (iii).

(iv).
We first compute the cases $q_1 = 1 = M_1$, $q_2 = 2 = M_2$,
$q_3 = 1 = \lfloor M_3 \rfloor$. Now take $n \geq 4$ and set $m=n+1$ again.
We will refine \eqref{EQ:pset-1} to show that
\begin{equation} \label{EQ:pset-2}
  \max( \PSet_m ) \leq M_n.
\end{equation}
Note that $M_n \geq 2$ for $n \geq 4$. Following the arguments of part~(iii)
and noting that $\max(\PSet_{m,1}) \leq 2$, the inequality \eqref{EQ:pset-2}
evidently turns into
\begin{equation} \label{EQ:pset-estimate}
  \max( \PSet_{m,k} ) \leq M_n \quad (k = 2, 4, \dotsc, m - \delta_m),
\end{equation}
where
\begin{equation} \label{EQ:def-delta}
  \delta_m \DefEqDisp
    \begin{cases}
      \, 1, &\text{if $m$ is odd,} \\
      \, 2, &\text{if $m$ is even.}
    \end{cases}
\end{equation}

\textit{Case $m$ odd.}
We apply Lemma~\ref{LEM:pset_odd} to get $\frac{m+1}{2} = \frac{n+2}{2} = M_n$.

\textit{Case $m$ even.}
Lemma~\ref{LEM:pset_even} yields $\frac{m+1}{3} = \frac{n+2}{3} = M_n$.

\noindent Both cases establish \eqref{EQ:pset-estimate} and, consequently,
\eqref{EQ:pset-2}. By \eqref{EQ:q_n-prod}, this shows the bounds in~(iv).
This completes the proof of Theorem~\ref{THM:main}.
\end{proof}


\section{Further properties} \label{SEC:properties}
Here we give an intermediate result that shows another formula for the values of
$q_n$, which we need later on.

\begin{theorem} \label{THM:formula-eps}
Let $q_n$ and $M_n$ be defined as in Theorem \ref{THM:main}.
For any fixed $n \geq 0$ we have
\begin{equation} \label{EQ:prod-q_n-eps}
  q_n = \hspace{-0.1em} \prod_{p \, \leq \, M_n}
        \hspace{-0.1em} p^{\varepsilon_p},
\end{equation}
where $\varepsilon_p$\ \!, which depends on $n$, is defined for $p=2$ by
$$
  \varepsilon_2 \DefEq
  \begin{cases}
    \, 0, & \text{if $n=2^r-1$ for some $r \geq 0$,} \\
    \, 1, & \text{otherwise,}
  \end{cases}
$$
and for an odd prime $p \leq M_n$ by
$$
  \varepsilon_p \DefEq
  \begin{cases}
    \, 0, & \text{if $p \nmid (n+2)$ and}\
            p \mid \binom{n+1}{j(p-1)}\ \text{for all}\
            j = 2,3, \dotsc, \left\lfloor \frac{n}{p-1} \right\rfloor - 1, \\
    \, 1, &\text{otherwise.}
  \end{cases}
$$
\end{theorem}

For a refinement step in the proof of Theorem~\ref{THM:formula-eps}, we need
the following congruence. See Hermite \cite{hermite} for the case $m$ odd, and
Bachmann \cite[Eq.~(116), p.~46]{bach} for the general case. For a recent
elementary proof, see \cite{m&s}.

\begin{lemma}[Hermite, Bachmann]
If $p$ is a prime and $m \geq 1$, then
\begin{equation} \label{EQ:hermite}
  \sum_{1 \, \leq \, j \, \leq \, \frac{m-1}{p-1}}
  \binom{m}{j(p-1)} \equiv 0 \pmod{p}.
\end{equation}
\end{lemma}

\begin{proof}[Proof of Theorem~\ref{THM:formula-eps}]
By Theorem~\ref{THM:main} we know that $q_n$ is the product of certain primes
$p \leq M_n$. Thus, to show \eqref{EQ:prod-q_n-eps}, we have to determine
the exponents $\varepsilon_p$.

If $p=2$, then the value of $\varepsilon_2$ is given by Theorem~\ref{THM:main}
part~(iii) for $n \geq 1$. Since $M_0 = 1$, the case $n=0$ does not occur here.
So we are done in case $p=2$.

We now consider the case of an odd prime $p\leq M_n$. Since $M_n \leq 2$ if
$n \leq 3$, we may fix $n \geq 4$. Set $m = n+1$ and recall the definition of
$\delta_m$ in \eqref{EQ:def-delta}. As in the proof of Theorem~\ref{THM:main}
and in relation~\eqref{EQ:q_n-prod}, we have that
$$
  p \nmid q_n
  \quad \iff \quad
  p \notin \PSet_{m}
$$
and by Lemma~\ref{LEM:pset_def} we obtain that
$$
  p \nmid q_n
  \quad \iff \quad
  p \notin \PSet_{m,k} \quad (k = 2, 4, \dotsc, m - \delta_m).
$$
Note that $p \nmid q_n$ is equivalent to $\varepsilon_p = 0$. Recall from
Lemma~\ref{LEM:pset_def} that $p \in \PSet_{m,k}$ if and only if $(p-1) \mid k$
and $p \nmid \binom{m}{k}$. As $m-\delta_m$ and $p-1$ are even, we have
$$
  L \DefEq \left\lfloor \frac{m - 1}{p-1} \right\rfloor
  = \left\lfloor \frac{m - \delta_m}{p-1} \right\rfloor.
$$
Substituting $k \mapsto j(p-1)$ for those $k$ with $(p-1) \mid k$,
we finally conclude that
\begin{equation} \label{EQ:cond-eps-0}
  \varepsilon_p = 0
  \quad \iff \quad
  p \mid \binom{m}{j(p-1)}
  \quad \left( j = 1, 2, \dotsc, L \right).
\end{equation}
We then infer by using \eqref{EQ:hermite} that
$$
  \sum_{1 \leq \, j \, < \, L} \binom{m}{j(p-1)} \equiv 0 \pmod{p}
  \quad \implies \quad
  \binom{m}{L(p-1)} \equiv 0 \pmod{p}.
$$
Thus, we may replace the last index $L$ with $L-1$ in \eqref{EQ:cond-eps-0}.
It remains to show in case $j=1$ that $p \mid \binom{m}{p-1}$ is equivalent to
$p \nmid (m+1)$. To see this, note first that
$$
  p \mid \binom{m}{p-1} \quad \iff \quad p \mid m(m-1) \dotsm (m-(p-2)).
$$
Now, since $m, m-1, \dotsc, m-(p-2),$ and $m+1$ represent the $p$ different
residue classes modulo~$p$, we deduce the desired equivalence.
This completes the proof of Theorem~\ref{THM:formula-eps}.
\end{proof}


\section{Proof of \texorpdfstring{Theorem~\ref{THM:formula}}{Theorem~2}} \label{SEC:proofs2}
Before giving the proof, we first introduce some notation and two preliminary
results.

For a prime $p$, the {\em $p$-adic valuation} of an integer $x > 0$, denoted by
$v_p(x)$, gives the exponent of the highest power of $p$ that divides $x$.
Any integer $x \geq 0$ can be written as a finite {\em $p$-adic expansion}
$$
  x = \alpha_0 + \alpha_1 \, p + \dotsb + \alpha_r \, p^r
$$
with some $r \geq 0$ and unique {\em base-$p$ digits} $\alpha_j$ satisfying
$0 \leq \alpha_j \leq p-1$ for $j = 0, 1, \dotsc, r$.
(In case $x > 0$, we assume that $\alpha_r > 0$, unless $r$ is prescribed,
when trailing zeros may occur.)
The sum of the digits of this expansion defines the function
$$
  s_p(x) \DefEq  \alpha_0 + \alpha_1 + \dotsb + \alpha_r.
$$
Note that $s_p(x)=0$ if and only if $x=0$. Comparing the two equations above,
one simply observes that
\begin{equation} \label{EQ:s_p-congr}
  s_p( x ) \equiv x \pmod{(p-1)}.
\end{equation}
A further property, proved by Legendre \cite[pp. 8--10]{legendre}
(see also \cite[p.~77]{moll}), is that
$$
  v_p( x! ) = \frac{x - s_p(x)}{p-1},
$$
also implying \eqref{EQ:s_p-congr} at once.
An easy application to binomial coefficients provides that
\begin{equation} \label{EQ:v_p-binom}
  v_p \left( \binom{m}{k} \right) = \frac{s_p(k) + s_p(m-k) - s_p(m)}{p-1}.
\end{equation}

We are ready now to prove our second main result.

\begin{proof}[Proof of Theorem~\ref{THM:formula}]
Fix $n \geq 0$ and set $m = n+1$. With the help of
Theorem~\ref{THM:formula-eps} and its proof, we will show that
\eqref{EQ:prod-q_n-eps} is equivalent to \eqref{EQ:prod-q_n}.
To do so, we have to show for all primes $p \leq M_n$ that
\begin{equation} \label{EQ:e_p-s_p}
  \varepsilon_p = 1 \quad \iff \quad s_p(m) \geq p.
\end{equation}

If $n < 2$, then $M_n=1$, and we are done. Now assume that $n \geq 2$, so that
$m \geq 3$. As in the proof of Theorem~\ref{THM:formula-eps}, we set
$$
  L \DefEq \left\lfloor \frac{m-1}{p-1} \right\rfloor.
$$

\textit{Case $\varepsilon_2$.}
Since $m \geq 3$, Theorem~\ref{THM:formula-eps} implies that $\varepsilon_2 = 0$
if and only if $m$ is a power of $2$. The latter is equivalent to $s_2(m) = 1$,
as well as to $m$ having only one digit equal to~$1$ in its binary expansion.
Thus if $\varepsilon_2 = 1$, then we must have $s_2(m) \geq 2$, and conversely.
This shows \eqref{EQ:e_p-s_p} for $p=2$.

\textit{Case $\varepsilon_p$ for odd $p \leq M_n$.} ``$\Rightarrow$'':
If $\varepsilon_p = 1$, then we deduce from \eqref{EQ:cond-eps-0} that there
exists a positive index $j \leq L$ such that
$$
  p \nmid \binom{m}{j(p-1)},
  \quad \text{that is}, \quad
  v_p \left( \binom{m}{j(p-1)} \right) = 0.
$$
Using \eqref{EQ:v_p-binom} we then obtain that
\begin{equation} \label{EQ:s_p-binom}
  s_p(m) = s_p(j(p-1)) + s_p(m-j(p-1)).
\end{equation}
As $j \geq 1$, we conclude by \eqref{EQ:s_p-congr} that
$s_p(j(p-1)) \geq p-1$. Since \mbox{$m > j(p-1)$} by $j \leq L$, we have
$s_p(m-j(p-1)) \geq 1$.
Applying these estimates to \eqref{EQ:s_p-binom}, we finally infer that
$s_p(m) \geq p$.

``$\Leftarrow$'':
We now suppose that $s_p(m) \geq p$. The bound $p \leq M_n$ implies that
\mbox{$m > p$}. Therefore, in the $p$-adic expansion
\begin{equation} \label{EQ:alpha_j}
  m = \alpha_0 + \alpha_1 \, p + \dotsb + \alpha_r \, p^r,
\end{equation}
we have $r \geq 1$ and $\alpha_r > 0$. Since $m > j(p-1)$ when $1 \leq j \leq L$,
the $p$-adic expansion
\begin{equation} \label{EQ:beta_j}
  j(p-1) = \beta_{j,0} + \beta_{j,1} \, p + \dotsb + \beta_{j,r} \, p^r
\end{equation}
has $\alpha_r \geq \beta_{j,r} \geq 0$ and the digits
$\beta_{j,0}, \dotsc, \beta_{j,r}$ cannot all be equal to the digits
$\alpha_0, \dotsc, \alpha_r$. By Lucas's theorem \cite[pp.~417--420]{lucas}
(for a modern proof, see \cite{fine} or \cite[p.~70]{moll}), we obtain
\begin{equation} \label{EQ:omega_j}
  \omega_j \DefEq
  \binom{m}{j(p-1)} \equiv
  \binom{\alpha_0}{\beta_{j,0}} \binom{\alpha_1}{\beta_{j,1}}
    \dotsm \binom{\alpha_r}{\beta_{j,r}} \pmod{p},
\end{equation}
using the convention that $\binom{\alpha}{\beta} = 0$ if $\alpha < \beta$.
We will deduce that $\omega_j \not\equiv 0 \pmod{p}$ for some index $j$.
To do so, we construct unique digits $\beta'_0, \dotsc, \beta'_r$, as follows.
(Remember that $r \geq 1$ and $\alpha_r \geq 1$.)
\begin{itemize}
\item Set $\beta'_r \DefEq \alpha_r - 1$.

\item Set $\beta'_k \DefEq \min \left( \alpha_k, \, (p-1)
      - \hspace*{-0.5em} \displaystyle\sum\limits_{\ell=k+1}^{r}
      \beta'_{\ell} \right)$
      iteratively for $k = r-1, r-2, \dotsc, 0$.
\end{itemize}

Roughly speaking, the digits $\beta'_k$ are ``filled up'' by the digits $\alpha_k$,
until the partial sum \mbox{$\beta'_{k+1} + \dotsb + \beta'_r$} reaches $p-1$;
the remaining $\beta'_k$ are then set equal to zero.

To explain this procedure in a more striking manner, imagine the following
picture. We take $p-1$ marbles, which we use to fill $r+1$  cups arranged in a
row and numbered $k=0, 1, \dotsc, r$. These cups, whose contents represent the
digits $\beta'_k$, are initially empty. (The actual procedure above omits this
step and iteratively sets each digit $\beta'_k$ to its final value.) We put
$\alpha_r - 1$ marbles into the cup with index $k=r$, while we fill the other
cups (successively having index $k=r-1, r-2, \dotsc, 0$) with up to
$\alpha_k$ marbles, if possible. We stop this process when we have used all the
marbles. (The actual procedure does not stop and sets all remaining $\beta'_k$
equal to zero.) In total, we have placed at most
$\alpha_r - 1 + \alpha_{r-1} + \dotsb + \alpha_0 = s_p(m)-1$,
but not exceeding $p-1$, marbles in the cups. Therefore, all $\beta'_k$ satisfy
$0 \leq \beta'_k \leq p-1$. It follows that if $s_p(m) \geq p$, then all $p-1$
marbles necessarily have been distributed over the cups.

Since $s_p(m) \geq p$, that is, $\alpha_0 + \dotsb + \alpha_r \geq p$,
we obtain the following properties:
\begin{enumerate}
\item $s_p(b) = p-1$, where
      $b \DefEq \beta'_0 + \beta'_1 \, p + \dotsb + \beta'_r \, p^r$.

\item $\alpha_k \geq \beta'_k$ for $k=0, 1, \dotsc, r-1$.

\item $\alpha_r > \beta'_r$.
\end{enumerate}
By using \eqref{EQ:s_p-congr}, property (i) implies that $(p-1) \mid b$.
From property (iii) and the expansion \eqref{EQ:alpha_j} we conclude that $b < m$.
Therefore, taking the index $j = b/(p-1)$, which satisfies $1 \leq j \leq L$, the digits
$\beta'_0, \dotsc, \beta'_r$ equal the digits $\beta_{j,0}, \dotsc, \beta_{j,r}$,
since $j(p-1) = b$, as used in \eqref{EQ:beta_j} and \eqref{EQ:omega_j}.
Furthermore, by properties (ii) and (iii), all binomial coefficients
$$
  \binom{\alpha_k}{\beta_{j,k}}
  = \binom{\alpha_k}{\beta'_k}
  \not\equiv 0 \pmod{p}
  \quad (k = 0, 1, \dotsc, r).
$$
Applying Lucas's theorem in \eqref{EQ:omega_j}, we finally achieve
$\omega_j \not\equiv 0 \pmod{p}$. By \eqref{EQ:cond-eps-0} this shows that
$\varepsilon_p = 1$.

All cases for $\varepsilon_2$ and $\varepsilon_p$ show that \eqref{EQ:e_p-s_p}
holds, proving the formula for $q_n$ in \eqref{EQ:prod-q_n}.

To complete the proof of Theorem~\ref{THM:formula}, it suffices to show that the
bound $M_n$ on the prime factors $p$ of $q_n$, given by Theorem~\ref{THM:main},
is sharp for infinitely many even (respectively, odd) values of $n$.
To see this, let $p$ be any odd prime and set $n_0 = 2p-2$. Then $n_0$ is even
and $M_{n_0}=p$. Since the $p$-adic expansion of $n_0$ is $n_0 = (p-2) + p$,
we have $s_p(n_0+1) = (p-1) + 1 = p$, ensuring that $p \mid q_{n_0}$.
A similar argument applied to odd $n_1 = 3p-2$ shows that $M_{n_1}=p$ and
$p \mid q_{n_1}$. Theorem~\ref{THM:formula} follows.
\end{proof}


\section{Applications}
The formulas for $q_n$ are intimately connected with the Bernoulli polynomials $B_n(x)$
by \eqref{EQ:def-S_n}. Therefore, we can reformulate Theorem~\ref{THM:formula}
in a way that describes the denominators of these polynomials.

\begin{theorem} \label{THM:formula-bernpoly}
For $n \geq 1$ let
$$
  \DD_n \DefEqDisp \denom \bigl( B_n(x) \bigr).
$$
The values $\DD_n$ have the following properties:
\begin{enumerate}
\item If $n=1$, then $\DD_n = 2$.

\item If $n \geq 3$ is odd, then
  $$
    \DD_n =
    \prod_{
      \substack{
        p \, \leq \, \frac{n+1}{2} \\[0.2em]
        s_p(n) \, \geq \, p}
    } p.
  $$

\item If $n \geq 2$ is even, then
  $$
    \DD_n =
  	\hspace{-0.2em}
  	\prod_{(p-1) \, \mid \, n}
  	\hspace{-0.2em} p
    \quad \times \quad
  	\hspace{-0.6em}
    \prod_{
      \substack{
        (p-1) \, \nmid \, n \\[0.2em]
        p \, \leq \, \frac{n+1}{3} \\[0.2em]
        s_p(n) \, \geq \, p}
    } \hspace{-0.2em} p.
  $$
\end{enumerate}
In particular, $\DD_n$ is even and squarefree for all $n \geq 1$,
and the sequence $(\DD_n)_{n \geq 1}$ is unbounded. Moreover,
\begin{equation} \label{EQ:bern-div-bernpoly}
  \denom \bigl( B_n \bigr) \mid \denom \bigl( B_n(x) \bigr) \quad (n \geq 1).
\end{equation}
\end{theorem}

The first few values of $\DD_n$ (see \cite[Sequence A144845]{oeis}) are
$$
  \DD_n = 2, 6, 2, 30, 6, 42, 6, 30, 10, 66, 6, 2730, 210, 30, 6, 510, 30, 3990, 210, \dotsc.
$$

\begin{proof}[Proof of Theorem~\ref{THM:formula-bernpoly}]
As a result of Theorem~\ref{THM:main}, we obtain by \eqref{EQ:def-S_n} and
\eqref{EQ:denom-S_n} that
\begin{equation} \label{EQ:denom-bernpoly}
  q_{n-1} = \denom \bigl( B_n(x) - B_n \bigr) \quad (n \geq 1).
\end{equation}

(i).
Since $B_1(x) = x - \tfrac{1}{2}$ by \eqref{EQ:def-B_nx} and \eqref{EQ:B0B1},
we get $\DD_1 = 2$, which is even and squarefree. Then from $B_1 = -\tfrac{1}{2}$
by \eqref{EQ:B0B1}, relation~\eqref{EQ:bern-div-bernpoly} holds for $n=1$.
This shows~(i).

(ii).
Let $n \geq 3$ be odd. Then $B_n = 0$, so by \eqref{EQ:denom-bernpoly} we obtain
$$
  \DD_n  = \denom \bigl( B_n(x) \bigr) = q_{n-1}.
$$
The squarefree product formula for $\DD_n$ follows by applying
Theorem~\ref{THM:formula} to $q_{n-1}$. Since $n \geq 3$ is odd, $q_{n-1}$ must
be even by Theorem~\ref{THM:main} part~(iii). As $\denom(B_n) = 1$ in this case,
\eqref{EQ:bern-div-bernpoly} trivially holds. This shows (ii).

(iii).
Let $n \geq 2$ be even. Recall the von Staudt--Clausen theorem in
\eqref{EQ:denom-B_n}, namely,
\begin{align} \label{EQ:D_n}
  D_n \DefEqDisp \denom( B_n ) =
	\hspace{-0.3em}
	\prod_{(p-1) \, \mid \, n}
	\hspace{-0.3em} p.
\end{align}
Since $B_n(x) - B_n$ has no constant term (see the proof of
Lemma~\ref{LEM:poly-T_n}), we deduce that
\begin{align*}
  \DD_n &= \denom \bigl( (B_n(x) - B_n) + B_n \bigr) \\
    &= \lcm \bigl( \denom( B_n(x) - B_n ), \denom( B_n ) \bigr) \\
    &= \lcm( q_{n-1}, D_n ).
\end{align*}
Thus, $D_n \mid \DD_n$, which shows \eqref{EQ:bern-div-bernpoly}.
As $2 \mid D_n$ by \eqref{EQ:D_n}, we then infer that $\DD_n$ is even. Since
$D_n$ and $q_{n-1}$ are  squarefree, so is $\DD_n$. Finally, we get
$$
  \DD_n = D_n \times \frac{q_{n-1}}{\gcd(q_{n-1},D_n)},
$$
where the second factor does not include primes that divide $D_n$.
By Theorem~\ref{THM:formula}, the result follows. This proves (iii).

It remains to show that the sequence $(\DD_n)_{n \geq 1}$ is unbounded.
Since parts~(ii) and~(iii) imply that $q_{n-1} \mid \DD_n$ for $n \geq 2$,
Theorem~\ref{THM:formula} gives the result again.
\end{proof}

To put Theorem~\ref{THM:formula-bernpoly} in the context of known results,
we note a special property of the values of the Bernoulli polynomials at
rational arguments, namely,
\begin{equation} \label{EQ:bernpoly-Z}
  k^n \left( B_n \left( \frac{h}{k} \right) - B_n \right) \in \Z
    \quad (k \in \N, \, h \in \Z).
\end{equation}
This result is due to Almkvist and Meurman \cite{am}; for a different proof,
see \cite[pp.~70--71]{cohen}. As a complement, from \eqref{EQ:denom-bernpoly}
we have
\begin{equation} \label{EQ:bernpoly-Zx}
  q_{n-1} \bigl( B_n(x) - B_n \bigr) \in \Z[x].
\end{equation}
We argue that relations \eqref{EQ:bernpoly-Z} and \eqref{EQ:bernpoly-Zx} are
independent. On the one hand, \eqref{EQ:bernpoly-Zx} at once implies
\eqref{EQ:bernpoly-Z} but with an extra factor $q_{n-1}$. To see this, set
$x = h/k$ in \eqref{EQ:bernpoly-Zx}, multiply by $k^n$, and recall that $B_n(x)$
is a polynomial of degree $n$.
On the other hand, $\eqref{EQ:bernpoly-Z}$ holds when $k$ is any prime,
whether or not it divides the denominator $q_{n-1}$ in \eqref{EQ:bernpoly-Zx}.

It is an astonishing fact that the denominators of $S_n(x)$ and $B_n(x) - B_n$
can be easily computed, without knowledge of the Bernoulli numbers, from the
formulas in Theorems~\ref{THM:main} and~\ref{THM:formula}, giving a link to
$p$-adic theory via the function $s_p(n)$.
By contrast, the formula for the denominator of $B_n(x)$ in
Theorem~\ref{THM:formula-bernpoly} part~(iii) is more complicated, being
separated into two products and requiring the von Staudt--Clausen theorem.

It is quite remarkable that surprising new properties of the power sum $S_n(x)$
are still being revealed four centuries after 1614. Already in that year,
in his book {\em Newer Arithmetischer Wegweyser} \cite{faul1614}, Faulhaber
published formulas he had initially found up to $n=7$ (see \cite{beery, schneider1983}),
extending the classical formulas for $n=1,2,3,4$.


\section*{Acknowledgment}
The authors thank Kieren MacMillan for suggestions on an early version of the paper.
We are grateful to the Archive of the City of Ulm for its kind support in
providing a digital copy of Sebastian Furck's portrait of Johann Faulhaber
\cite{furck} in Figure~\ref{FIG:Faulhaber}, with permission to reprint it.
We also appreciate the valuable suggestions of the referees which improved the paper.


\end{document}